\documentclass[twoside, 11pt]{article}
\usepackage{amssymb, amsmath, mathrsfs, amsthm}
\usepackage{graphicx}
\usepackage{color}
\usepackage[top=2cm, bottom=2cm, left=2cm, right=2cm]{geometry}
\usepackage{float, caption, subcaption, wrapfig}

\input{epsf}

\newcommand{\bd}{\begin{description}}
\newcommand{\ed}{\end{description}}
\newcommand{\bi}{\begin{itemize}}
\newcommand{\ei}{\end{itemize}}
\newcommand{\be}{\begin{enumerate}}
\newcommand{\ee}{\end{enumerate}}
\newcommand{\beq}{\begin{equation}}
\newcommand{\eeq}{\end{equation}}
\newcommand{\beqs}{\begin{eqnarray*}}
\newcommand{\eeqs}{\end{eqnarray*}}

\newcommand{\ceil}[1]{\left\lceil #1 \right\rceil}

\definecolor{DarkGreen}{rgb}{0.2, 0.6, 0.3}

\newtheorem{theorem}{Theorem}

\newtheorem{lemma}{Lemma}

\newtheorem{case}{Case}
\newtheorem{subcase}{Subcase}[case]
\newtheorem{claim}{Claim}
\newtheorem{fact}{Fact}
\newtheorem{proposition}{Proposition}

\begin{document}
\title{\textbf{Gallai-Ramsey number for the union of stars\footnote{Supported by the National Science Foundation of China
        (Nos. 11601254, 11551001, 11161037, and 11461054) and the Science
        Found of Qinghai Province (Nos. 2016-ZJ-948Q).}}}

\author{
Yaping Mao\footnote{School of Mathematics and Statistics, Qinghai
Normal University, Xining, Qinghai 810008, China. {\tt
maoyaping@ymail.com}} \footnote{Academy of Plateau Science and Sustainability, Xining, Qinghai 810008, China}, \ \ Zhao Wang\footnote{College of Science,
China Jiliang University, Hangzhou 310018, China. {\tt
wangzhao@mail.bnu.edu.cn}}, \ \  Colton Magnant\footnote{Department
of Mathematics, Clayton State University, Morrow, GA, 30260, USA.
{\tt dr.colton.magnant@gmail.com}} \footnotemark[2], \ \  Ingo
Schiermeyer\footnote{Technische Universit{\"a}t Bergakademie
Freiberg, Institut f{\"u}r Diskrete Mathematik und Algebra, 09596
Freiberg, Germany. {\tt Ingo.Schiermeyer@tu-freiberg.de}}. }

\date{}
\maketitle

\begin{abstract}
Given a graph $G$ and a positive integer $k$, define the
\emph{Gallai-Ramsey number} to be the minimum number of vertices $n$
such that any $k$-edge coloring of the complete graph
$K_n$ contains either a rainbow
(all different colored) triangle or a monochromatic copy of $G$. In
this paper, we obtain the exact value of the Gallai-Ramsey numbers for
the union of two stars in many cases and bounds in other cases. This
work represents the first class of disconnected graphs to be
considered as the desired monochromatic subgraph.
\end{abstract}

\section{Introduction}

In this work, we consider only edge-colorings of graphs. A coloring
of a graph is called \emph{rainbow} if no two edges have the same
color.

Colorings of complete graphs that contain no rainbow triangle have
very interesting and rather surprising structure. In 1967, Gallai
\cite{MR0221974} first examined this structure using the terminology
transitive orientations. The result was reproven in \cite{MR2063371}
in the terminology of graphs and can also be traced to
\cite{MR1464337}. For the following statement, a trivial partition
is a partition into only one part.

\begin{theorem}[\cite{MR1464337,MR0221974,MR2063371}]\label{Thm:G-Part}
In any coloring of a complete graph containing no rainbow triangle, there exists a nontrivial partition of the vertices (that is, with at least two parts) such that there are at most two colors on the edges between the parts and only one color on the edges between each pair of parts.
\end{theorem}

For ease of notation, we refer to a colored complete graph with no
rainbow triangle as a \emph{Gallai-coloring} and the partition
provided by Theorem~\ref{Thm:G-Part} as a \emph{Gallai-partition}.
The induced subgraph of a Gallai colored complete graph constructed
by selecting a single vertex from each part of a Gallai partition is
called the \emph{reduced graph} of that partition. By
Theorem~\ref{Thm:G-Part}, the reduced graph is a $2$-colored
complete graph.

Given two graphs $G$ and $H$, let $R(G, H)$ denote the $2$-color
Ramsey number for finding a monochromatic $G$ or $H$, that is, the
minimum number of vertices $n$ needed so that every red-blue
coloring of $K_{n}$ contains either a red copy of $G$ or a blue copy
of $H$. Although the reduced graph of a Gallai partition uses only
two colors, the original Gallai-colored complete graph could
certainly use more colors. With this in mind, we consider the
following generalization of the Ramsey numbers. Given two graphs $G$
and $H$, the \emph{general $k$-colored Gallai-Ramsey number} ${\rm
gr}_k(G:H)$ is defined to be the minimum integer $m$ such that every
$k$-coloring of the complete graph on $m$ vertices contains either a
rainbow copy of $G$ or a monochromatic copy of $H$. With the
additional restriction of forbidding the rainbow copy of $G$, it is
clear that ${\rm gr}_k(G:H)\leq {\rm R}_k(H)$ for any graph $G$.

We refer the interested reader to \cite{MR1670625} for a dynamic
survey of small Ramsey numbers and \cite{FMO14} for a dynamic survey
of rainbow generalizations of Ramsey theory, including topics like
Gallai-Ramsey numbers. One may notice that all of the results
contained in the dynamic survey regarding Gallai-Ramsey numbers
consider monochromatic subgraphs that are connected. In this work,
we consider the Gallai-Ramsey numbers for finding either a rainbow
triangle or monochromatic copy of $K_{1, n} \cup K_{1, m}$, this
being the first examination of disconnected monochromatic subgraphs.

The star was one of the first monochromatic graphs to be considered
in the context of Gallai-Ramsey numbers.

\begin{theorem}{\upshape \cite{FM11}}\label{Thm:K3-Star}
For $m \geq 2$ and $k \geq 2$,
$$
{\rm gr}_{k}(K_{3} : K_{1, m}) = \begin{cases}
\frac{5m - 6}{2} & \text{ if $m$ is even,}\\
\frac{5m - 3}{2} & \text{ if $m$ is odd.}
\end{cases}
$$
\end{theorem}

Our first result, provided in Section~\ref{Sec:StarLemma}, provides
an analysis of the stability of Theorem~\ref{Thm:K3-Star}.

For the union of two stars, Grossman \cite{Grossman} obtain the classical Ramsey number in the following result.
\begin{theorem}[\cite{Grossman}]\label{Thm:RamseyUnionStars}
Let $n,m$ be two integers with $n\geq m\geq 1$, and let
$K_{1,n},K_{1,m}$ be two stars. Then
$$
{\rm R}(K_{1,n}\cup K_{1,m})=\max\{n+2m, 2n+1,n+m+3\}.
$$
\end{theorem}

In Section~\ref{Sec:Diff}, we prove the exact value of the Gallai-Ramsey number for the union of two stars $K_{1, n} \cup K_{1, m},$ where $m\leq \frac{n-8}{6}$.\\
\begin{theorem}\label{Thm:GallaiRamseySmallm}
Let $n\geq 22$, $m \geq 5$ and $k \geq 3$ be three integers with
$m\leq \frac{n-8}{6}$. Then
$$
{\rm gr}_{k}(K_{3}:K_{1,n}\cup K_{1,m})=\begin{cases}
\frac{5n-6}{2}+k-3 & \text{ if $n$ is even,}\\
\frac{5n-3}{2}+k-3 & \text{ if $n$ is odd.}
\end{cases}
$$
\end{theorem}

In Section~\ref{Sec:Equal}, we prove the exact value of the
Gallai-Ramsey number for the union of two equal stars $K_{1, n} \cup
K_{1, n}$.
\begin{theorem}\label{Thm:GallaiRamseyLargem}
For $k \geq 3$,
$$
{\rm gr}_{k}(K_{3} :K_{1,n}\cup K_{1,n}) = 3n+k-1.
$$
\end{theorem}

Finally, in Section~\ref{Sec:General}, we get the upper and lower bounds for the union of
two stars for general $n,m$.
\begin{theorem}\label{Thm:GallaiRamseyGeneralCase}
Let $n\geq 9$, $n>m \geq 2$ and $k \geq 3$ be three integers with
$m\geq \frac{n-2}{6}$. Then
$$
\begin{cases}
\max\left\{2n+m+k-5,\frac{5n-6}{2}+k-3\right\}\leq {\rm gr}_{k}(K_{3}:K_{1,n}\cup K_{1,m})\leq 3n+3m+k-3 & \text{ if $n$ is even,}\\
\max\left\{2n+m+k-4,\frac{5n-3}{2}+k-3\right\}\leq {\rm
gr}_{k}(K_{3}:K_{1,n}\cup K_{1,m})\leq 3n+3m+k-2 & \text{ if $n$ is
odd.}
\end{cases}
$$
\end{theorem}

\section{Star lemma}\label{Sec:StarLemma}

First a helpful lemma regarding the stability of
Theorem~\ref{Thm:GallaiRamseySmallm}.

\begin{lemma}\label{Lemma:StarStability}
For positive integers $n$ and $r$ with $n \geq 22$ and $4 \leq r
\leq \frac{n + 4}{4}$, if $G$ is a Gallai coloring of a complete
graph of order $\frac{5n - r}{2}$ which contains no monochromatic
copy of $K_{1, n}$ on edges between parts of a Gallai partition,
then there are exactly $5$ parts in any Gallai partition, each of
order at least $\frac{n - r + 3}{2}$, and every vertex has at least
$n - r + 3$ incident edges to other parts of the Gallai partition in
each of two colors (the two colors appearing between parts of the
Gallai partition).
\end{lemma}

\begin{proof}
Let $G$ be a Gallai coloring of a complete graph of order $\frac{5n
- r}{2}$ and suppose that $G$ contains no monochromatic copy of
$K_{1, n}$. By Theorem~\ref{Thm:G-Part}, there is a Gallai partition
of $G$, say using red and blue on edges between parts of the
partition. Choose such a partition $H_1,H_2,\ldots,H_t$ with the
smallest number of parts $t$. In order to avoid a vertex of degree
$n$ in red or blue, certainly no part of the Gallai partition can
have order at least $n$. On the other hand, if we choose a vertex
$v$ in a part of order at most $\frac{n - r + 2}{2}$, then there are
at least
$$
\ceil{ \frac{|G| - \frac{n - r + 2}{2}}{2} } = \ceil{ \frac{ \frac{5n - r}{2} - \frac{n - r + 2}{2}}{2} } \geq n
$$
edges of one color incident to $v$, which means that every part of
the Gallai partition has order at least $\frac{n - r + 3}{2}$.

If $2 \leq t \leq 3$, then by minimality of $t$, we may assume that
$t = 2,$ but then there is a part of the partition of order at least
$\frac{|G|}{2} > n$, a contradiction. We may therefore assume that
$t \geq 4$. Additionally, since each part has order at least
$\frac{n - r + 3}{2}$ and $r \leq \frac{n + 4}{4}$, if any vertex
has edges of all one color to $3$ different parts, then there would
be a monochromatic copy of $K_{1, n}$. Thus, to avoid a
monochromatic copy of $K_{1, 3}$ in the reduced graph, we have $t
\leq 5$.

If $t = 4$, then there exists a ``big'' part of order at least
$\frac{|G|}{4} = \frac{5n - r}{8}$. By the minimality of $t$, the
four parts must have the structure that there is a red path
$H_3H_2H_1H_4$ and a blue path $H_1H_3H_4H_2$ in the reduced graph.
Since $|H_2|+|H_3|\geq n$ or $|H_1|+|H_4|\geq n$, it follows that
the induced subgraph by the edges from $H_2\cup H_3$ to $H_1\cup
H_4$ contains a blue star $K_{1,n}$, a contradiction.

Finally assume $t =5$. Then in order to avoid a monochromatic copy
of $K_{1, 3}$ in the reduced graph, the reduced graph must be the
unique $2$-coloring of $K_{5}$ containing no monochromatic triangle.
Thus, every vertex of $G$ has edges to exactly $2$ parts of the
Gallai partition in red and $2$ parts of the partition in blue. This
means that every vertex has red degree and blue degree at least
$2\cdot \frac{n - r + 3}{2} = n - r + 3$.
\end{proof}


\section{For small $m$ and large $n$}\label{Sec:Diff}

In this section we give a proof for Theorem \ref{Thm:GallaiRamseySmallm}.

For small $m$ and large $n$, we first give the lower bound on the
Gallai-Ramsey number for $K_{1,n}\cup K_{1,m}$.

\begin{lemma}\label{Lemma:Small_m_Large_n_Lower}
Let $n\geq 22$, $m \geq 5$ and $k \geq 3$ be three integers. Then
$$
{\rm gr}_{k}(K_{3} :K_{1,n}\cup K_{1,m}) \geq \begin{cases}
\frac{5n-6}{2}+k-3 & \text{ if $n$ is even,}\\
\frac{5n-3}{2}+k-3 & \text{ if $n$ is odd.}
\end{cases}
$$
\end{lemma}

\begin{proof}
We prove this result by inductively constructing a coloring of $K_{N}$ where
$$
N = \begin{cases}
\frac{5n-8}{2}+k-3 & \text{ if $n$ is even,}\\
\frac{5n-5}{2}+k-3 & \text{ if $n$ is odd,}
\end{cases}
$$
which contains no rainbow triangle and no monochromatic copy of
$K_{1,n}\cup K_{1,m}$.

For odd $n$, we construct $G_{3}^{o}$ by making five copies of
$K_{\frac{n-1}{2}}$ each colored entirely with color $1$, and then
inserting edges of colors $2$ and $3$ between the copies to form a
blow-up of the unique $2$-colored $K_{5}$ which contains no
monochromatic triangle. For even $n$, we construct $G_{3}^{e}$ by
making one copy of $K_{\frac{n}{2}}$ and four copies of
$K_{\frac{n-2}{2}}$ each colored entirely with color $1$, and then
inserting edges of colors $2$ and $3$ between the copies to form a
blow-up of the unique $2$-colored $K_{5}$ which contains no
monochromatic triangle. This coloring clearly contains no rainbow
triangle and, since no vertex has at least $n$ incident edges in any
one color, there can be no monochromatic copy of $K_{1,n}\cup
K_{1,m}$.

To this base graph, for each $i$ with $4 \leq i \leq k$ in sequence, we add a vertex $v_i$ with all edges to the new vertex $v_i$ having color $i$. The resulting colored complete graph has order $N$ and contains no rainbow triangle or monochromatic $K_{1,n}\cup K_{1,m}$, completing the construction.
\end{proof}

\begin{proposition}\label{Prop:Small_m_Large_n_Upper}
Let $n\geq 22$, $m \geq 5$ and $k \geq 3$ be three integers with
$m\leq \frac{n-8}{6}$. Then
$$
{\rm gr}_{k}(K_{3}:K_{1,n}\cup K_{1,m})=\begin{cases}
\frac{5n-6}{2}+k-3 & \text{ if $n$ is even,}\\
\frac{5n-3}{2}+k-3 & \text{ if $n$ is odd.}
\end{cases}
$$
\end{proposition}

\begin{proof}
Let $G$ be a Gallai-coloring of $K_{N}$ where
$$
N=\begin{cases}
\frac{5n-6}{2}+k-3 & \text{ if $n$ is even,}\\
\frac{5n-3}{2}+k-3 & \text{ if $n$ is odd.}
\end{cases}
$$
We only give the proof for the case that $n$ is odd, and the proof
of the case that $n$ is even can be proved similarly.

Since $G$ is a Gallai-coloring, by Theorem~\ref{Thm:G-Part}, there
is a Gallai-partition of $G$. Suppose red and blue are the two
colors appearing in the partition. Let $t$ be the number of parts in
this partition and choose such a partition where $t$ is minimized.

We consider the case when $n$ is odd since the case where $n$ is
even is similar. Let $r$ be the number of parts of the
Gallai-partition with order at least $\frac{n-1}{2}$, say with
$|H_{1}| \geq |H_{2}| \geq \dots \geq |H_{r}| \geq \frac{n-1}{2}$
and $|H_{r+1}|, |H_{r+2}|, \dots, |H_{t}|\leq \frac{n-3}{2}$.

The overall structure of the proof is by induction on $k$. We
consider cases based on the values of $k$ and $t$. For the base of
the induction, we first suppose that $k=3$ and further break into
two cases based on the value of $t$.

\setcounter{case}{0}
\begin{case}\label{Case:k3t2}
$k = 3$ and $2 \leq t \leq 3$.
\end{case}

Since $2\leq t\leq 3$, by the minimality of $t$, we may assume
$t=2$. Let $H_1$ and $H_2$ be the corresponding parts. Suppose all
edges from $H_1$ to $H_2$ are red. If $|H_1|\geq 2$ and $|H_2|\geq
2$, then there is a red $K_{1,n}\cup K_{1,m}$ since $|H_1|+|H_2| = N
= \frac{5n-3}{2}$ and so $|H_1|\geq
\frac{5n-3}{4}=n+\frac{n-3}{4}\geq n+m$, a contradiction since we
can build a red copy of $K_{1, n} \cup K_{1, m}$ with both stars
centered in $H_{2}$.

Suppose then that $|H_2|=1$. Then $|H_1|=\frac{5n-5}{2}$. Recall
that red is the color appearing in the edges from $H_1$ to $H_2$.
Let $H_2=\{v_1\}$.

If there is a vertex $w \in H_{1}$ with $m + 1$ incident red edges,
then there is a red copy of $K_{1, n} \cup K_{1, m}$ centered at $w$
and $v_{1}$. We may therefore assume that every vertex in $H_{1}$
has at most $m$ incident red edges. Choose a vertex $v_{2} \in
H_{1}$ with the smallest number of incident red edges. Let $S
\subseteq H_{1}$ be the set of vertices with red edges to $v_{2}$,
$T \subseteq H_{1}$ be the set of vertices with blue edges to
$v_{2}$, and $U \subseteq H_{1}$ be the set of vertices with green
edges to $v_{2}$. Then we have assumed $|S| \leq m - 1$ and suppose
$|T| \geq |U|$.

\begin{claim}\label{Claim:U}
$|U| \leq |T| \leq n + m$.
\end{claim}
\begin{proof}
Assume, to the contrary, that $|U| \geq n + m + 1$. If there is a
vertex $w \in S \cup T \cup U$ with a total of at least $m + 1$
incident blue edges, then there is a blue copy of $K_{1, n} \cup
K_{1, m}$ centered at $w$ and $v_{2}$, so every vertex in $S \cup T
\cup U$ has at most $m$ incident blue edges. This means that every
vertex in $S \cup T \cup U$ has at least
$$
|G| - 1 - 2m \geq \frac{5n - 3}{2} - 1 - 2m \geq n + m + 1
$$
incident green edges. This means that any two vertices in $S \cup T
\cup U$ form the centers of a green copy of $K_{1, n} \cup K_{1,
m}$, a contradiction.
\end{proof}

From Claim \ref{Claim:U}, we have $|U| \leq |T| \leq n + m$. Since
$|S| + |T| + |U| = N - 2$, we have $|T| \geq \frac{N - 2 - (m -
1)}{2} \geq n + 1$ and $|U| \geq N - n - 2m - 1$. If a vertex $w \in
T$ has at least $m$ blue edges to $U$, then there is a blue copy of
$K_{1, n} \cup K_{1, m}$ with centers $v_{2}$ and $w$. Thus, each
vertex $w \in T$ has at most $m - 1$ blue edges to $U$ and of course
$w$ has at most $m - 1$ red edges to $U$. This means that $w$ has at
least $|U| - 2(m - 1) \geq N - 4m - n + 1$ green edges to $U$, so
there are at total of at least $|T|(N - 4m - n + 1)$ green edges
between $T$ and $U$. This means there is a vertex $w \in U$ with at
least $\frac{|T|}{|U|}(N - 4m - n + 1) \geq N - 4m - n + 1$ green
edges to $T$. Within $U$, $w$ has at least $|U| - 1 - 2(m - 1)$
incident green edges so $w$ has a total of at least
$$
[N - 4m - n + 1] + [|U| - 2m + 1] \geq m + n + 1
$$
incident green edges. Then $w$ and $v_{2}$ form the centers of a green copy of $K_{1, n} \cup K_{1, m}$.

\begin{case}\label{Case:k3t4}
$k = 3$ and $t\geq 4$.
\end{case}

First a claim about the orders of parts in the Gallai partition.

\begin{claim}\label{Claim:AllBigger}
Every part has order at least $\frac{n - 5}{2}$.
\end{claim}

\begin{proof}
First suppose there is a part $H_{i}$ with $2 \leq |H_{i}| \leq
\frac{n - 4m - 6}{2}$. Let $A$ be the set of vertices in $G
\setminus H_{i}$ with red edges to $H_{i}$ and $B$ be the set of
vertices in $G \setminus H_{i}$ with blue edges to $H_{i}$ and
suppose that $|A| \geq |B|$. Then $|A| \geq \frac{N - |H_{i}|}{2}
\geq n + m$. This means that any two vertices of $H_{i}$ form the
centers of a red copy of $K_{1, n} \cup K_{1, m}$, a contradiction.
This means we may assume that every part has order either $1$ or at
least $\frac{n - 4m - 5}{2}$.

Suppose there is a part $H_{0}$ of order $1$, say with $H_{0} =
\{v_{1}\}$. Since $G$ cannot be colored entirely with $2$ colors,
there must also be a part $H_{1}$ with $|H_{1}| \geq \frac{n - 4m -
5}{2} \geq m + 1$. Let $A$ be the set of vertices with red edges to
$v_{1}$ and $B$ be the set of vertices with blue edges to $v_{1}$,
again assuming that $|A| \geq |B|$. In particular, this means that
$|A| \geq \frac{(5n - 6)/2 - 1}{2} \geq n + m + 1$ so every vertex
of $G$ other than $v_{1}$ has at most $m$ incident red edges. This
means that every part of order at least $m + 1$ has all blue edges
to other parts (except to $v_{1}$). If $H_1$ is in $B$, then to
avoid a red copy of $K_{1, n} \cup K_{1, m}$, the edges from $A$ to
$H_1$ are blue, and hence there is a blue copy of $K_{1, n} \cup
K_{1, m}$, a contradiction. So $H_1$ is in $A$. In order to avoid
making a blue copy of $K_{1, n} \cup K_{1, m}$ centered at one
vertex in $H_{1}$ and one vertex in $G \setminus (H_{1} \cup
\{v_{1}\})$, we must have $|G \setminus (H_{1} \cup \{v_{1}\})| < n
+ 1$ but then $|H_{1}| \geq n + m + 1$, in order to avoid making a
blue copy of $K_{1, n} \cup K_{1, m}$ centered at two vertices of $G
\setminus (H_{1} \cup \{v_{1}\})$, we must have $|G \setminus (H_{1}
\cup \{v_{1}\})| = 1$, which contradicts to the fact $t\geq 4$. This
means that every part has order at least $\frac{n - 4m - 5}{2} \geq
m + 1$.

Finally suppose there is a part $H_{i}$ with $m + 1 \leq |H_{i}|
\leq \frac{n - 7}{2}$. Then letting $A$ be the set of vertices with
red edges to $H_{i}$ and $B$ be the set of vertices with blue edges
to $H_{i}$, we see that one of $A$ or $B$ (suppose $A$) has order at
least $\frac{N - |H_{i}|}{2} \geq n + 1$. Then there is a red copy
of $K_{1, n} \cup K_{1, m}$ centered at one vertex in $H_{i}$ and
another in $A$, a contradiction.
\end{proof}

Let $r$ be the number of parts of order at least $\frac{n - 1}{2}$
and call these parts ``big''. Call any remaining parts ``small''.
Next a claim about the small parts. By Claim~\ref{Claim:AllBigger},
there are at most $5$ parts in this Gallai partition of $G$.

We distinguish the following subcases to complete the proof of this case.

\begin{subcase}\label{Subcase:r5}
$r=5$.
\end{subcase}

To avoid a pair of vertices having all one color on edges to three
big parts, the reduced graph on the parts $H_1,H_2,H_3,H_4,H_5$ must
be the unique $2$-coloring of $K_{5}$ with no monochromatic
triangle, say with $H_1H_2H_3H_4H_5H_1$ and $H_1H_3H_5H_2H_4H_1$
making two complementary monochromatic cycles with in red and blue
respectively. Since $\sum_{i=1}^5|H_i| = \frac{5n-3}{2}$, it follows
that there exists a big part, say $H_1$, such that $|H_1|\geq
\frac{n+1}{2}$. Choose $v_2\in H_2$ and $v_{4} \in H_{4}$. Then the
edges from $v_2$ to $H_1 \cup H_3$ contain a red copy of $K_{1,n}$,
and the edges from $v_4$ to $H_5$ contain a red copy of $K_{1,m}$,
and so there is a red $K_{1,n}\cup K_{1,m}$, a contradiction.

\begin{subcase}
$r=4$.
\end{subcase}

To avoid a part having edges of all one color to three of the big
parts, by symmetry, the four big parts must form one of the
following two structures:
\begin{itemize}
\item Type $1$: There is a red cycle $H_1H_3H_4H_2H_1$ and a blue $2$-matching $\{H_1H_4,H_2H_3\}$ in the reduced graph, or
\item Type $2$: There is a red path $H_3H_2H_1H_4$ and a blue path $H_1H_3H_4H_2$ in the reduced graph.
\end{itemize}

First suppose that $t=4$. Since $\sum_{i=1}^4|H_i|=\frac{5n-3}{2}$,
it follows that there exists a big part, without loss of generality
(regardless of Type~$1$ or Type~$2$) say $H_1$, such that $|H_1|\geq
\frac{n+3}{2}$. Choose $v_1 \in H_1$ and $v_2\in H_2$. Then the
edges from $v_2$ to $(H_1 \setminus \{v_1\})\cup H_3$ contain a red
copy of $K_{1,n}$ and the edges from $v_1$ to $H_4$ contain a red
copy of $K_{1,m}$, so there is a red copy of $K_{1,n}\cup K_{1,m}$,
a contradiction.

Thus, we may assume that $t = 5$. Let $H_1,H_2,H_3,H_4$ be the big
parts, and $H_5$ be the (small) part of order at least
$\frac{n-5}{2}$. If the reduced graph does not consist of two $5$-cycles
$H_1H_2H_3H_4H_5H_1$ and $H_1H_3H_5H_2H_4H_1$, then there exist
three parts, say $H_2,H_3,H_5$, adjacent to the part $H_1$ by a
single color. Since
$|H_2|+|H_3|+|H_5|=n-1+\frac{n-5}{2}=n+\frac{n-7}{2}\geq n+m$, this
structure contains a monochromatic copy of $K_{1, n} \cup K_{1, m}$,
a contradiction. Thus, suppose that the reduced graph is two
$5$-cycles $H_1H_2H_3H_4H_5H_1$ and $H_1H_3H_5H_2H_4H_1$ say in red
and blue respectively. By the same argument used in
Subcase~\ref{Subcase:r5}, where $r = t = 5$, there is a
monochromatic copy of $K_{1,n}\cup K_{1,m}$, a contradiction.

\begin{subcase}
$r\leq 3$.
\end{subcase}

Since $t \geq 4$, there is at least one small part, say $H_{1}$, so
$\frac{n - 5}{2} \leq |H_{1}| \leq \frac{n - 3}{2}$. Let $A$ be the
set of vertices with red edges to $H_{1}$ and let $B$ be the set of
vertices with blue edges to $H_{1}$. If $|H_{1}| = \frac{n - 5}{2}$,
then one of $|A|$ or $|B|$ is large, with say $|A| \geq \ceil{
\frac{N - \frac{n - 5}{2}}{2} } \geq n + 1$ so there is a blue copy
of $K_{n + 1, m + 1}$ between $H_{1}$ and $A$, which contains a blue
copy of $K_{1, n} \cup K_{1, m}$, a contradiction. Thus, all small
parts have order $\frac{n - 3}{2}$. In order to avoid the same
construction, we must have $|A| = |B| = n$.

Also since $t \leq 5$, there is at least one big part.

The following facts are also immediate from the restrictions on small parts.

\begin{fact}\label{Fact:n15}
One of the following holds:\\
$(1)$ $A$ contains only one big part with $|A|=n$, or\\
$(2)$ $A$ contains a small part $H_{1,A}$ with
$|H_{1,A}|=\frac{n-3}{2}$ and a big part $H_{2,A}$ with
$|H_{2,A}|=\frac{n+3}{2}$.
\end{fact}

\begin{fact}\label{Fact:n16}
One of the following holds:\\
$(3)$ $B$ contains only one big part with $|B|=n$, or\\
$(4)$ $B$ contains a small part $H_{1,B}$ with
$|H_{1,B}|=\frac{n-3}{2}$ and a big part $H_{2,B}$ with
$|H_{2,B}|=\frac{n+3}{2}$.
\end{fact}

If $(1)$ and $(3)$ both hold, then there are only three parts
$H_1,A,B$ which contradicts the assumption that $t\geq 4$.

First suppose (without loss of generality) that $(1)$ and $(4)$
hold. If the edges from $A$ to $H_{1,B}$ (or $H_{2,B}$) are blue,
then for $u, v \in H_{1, B}$ (respectively $H_{2, B}$), the edges
from $\{u, v\}$ to $H_1\cup A$ contain a blue copy of $K_{1,n}\cup
K_{1,m}$, a contradiction. Thus, the edges from $A$ to $H_{1,B}\cup
H_{2,B}$ must all be red. Then if we choose $u, v \in A$, then the
edges from $\{u, v\}$ to $H_1 \cup B$ contain a red copy of
$K_{1,n}\cup K_{1,m}$, a contradiction.

Finally assume that $(2)$ and $(4)$ hold. If all of the edges from
$H_{i,B}$ (with $i \in \{1, 2\}$) to $A$ are blue, then there is a
blue copy of $K_{1, n} \cup K_{1, m}$ centered at two vertices of
$H_{i, B}$, a contradiction. Thus, we may assume that for each part
$H_{i, B}$, there is a part $H_{j, A}$ such that the edges in
between these two parts are red where $i, j \in \{1, 2\}$. Similarly
from the opposite perspective, for each part $H_{j, A}$, there is a
part $H_{i, B}$ such that the edges between these two parts are blue
where $i, j \in \{1, 2\}$. From the above arguments, without loss of
generality, we may assume that

\begin{itemize}
\item the edges from $H_{1,A}$ to $H_{1,B}$ and the edges from $H_{2,A}$ to
$H_{2,B}$ are red;

\item the edges from $H_{1,A}$ to $H_{2,B}$ and the edges from $H_{2,A}$ to
$H_{1,B}$ are blue;
\end{itemize}

Now let $v \in H_{2, A}$ and let $u \in H_{1, A}$. Then the edges
from $u$ to $H_{2, B}$ contain a red copy of $K_{1, m}$ and the
edges from $v$ to $H_{1} \cup H_{2, B}$ contain a red copy of $K_{1,
n}$, a contradiction.

\begin{case}
$k\geq 4$.
\end{case}

Let $T$ be a largest set of vertices in $G$ such that each vertex in
$T$ has edges of all one color to $G \setminus T$ with the added
restriction that $|G \setminus T| \geq n + m$. For each $i$ with $1
\leq i \leq k$, if we let $T_{i}$ be the set of vertices in $T$ with
all edges of color $i$ to $G \setminus T$, then in order to avoid a
monochromatic copy of $K_{1, n} \cup K_{1, m}$, we have $|T_{i}|
\leq 1$ for all $i$. This means that $|T| \leq k$. Let $G' = G
\setminus T$ so $|G'| \geq \frac{5n - 12}{2}$.

Within $G'$, there is no vertex with degree at least $m$ in a color
$i$ for which $T_{i} \neq \emptyset$ to avoid creating a copy of
$K_{1, n} \cup K_{1, m}$ in color $i$. We first claim that there are
at least two colors not appearing on edges from $T$ to $G'$.

\begin{claim}\label{Claim12}
$|T| \leq k - 2$.
\end{claim}

\begin{proof}
First if $|T| = k$, then every vertex in $G'$ has color degree at
most $m - 1$ within $G'$ in every color. This contradicts
Theorem~\ref{Thm:K3-Star} since there must be a vertex with degree
at least $\frac{2|G'|}{5} \geq m$ in some color.

Next if $|T| = k - 1$, then there is again a vertex with degree at
least $m$ in some color (say red) but this means that red must be
the color not represented on edges between $T$ and $G'$. Consider a
Gallai partition of $G'$, say with the smallest possible number of
parts. If this partition has only $2$ parts, then red must be the
color between the parts and in order to avoid creating a red copy of
$K_{1, n} \cup K_{1, m}$, one part must have order $1$. This part
can be moved to $T$, contradicting the maximality of $|T|$. Thus, we
may assume that there are at least $4$ parts in the Gallai partition
of $G'$, say with red and blue appearing on edges between the
parts. By minimality of the number of parts in this partition, both
red and blue must induce connected subgraphs of the reduced graph.
This means that in order to avoid having a vertex with at least $m$
edges in blue, all parts of this partition must have order at most
$m - 1$. Every vertex $v \in G'$ must then have at least $|G| - 2(m
- 1) > n + m + 2$ incident red edges. This means that there is a red
copy of $K_{1, n} \cup K_{1, m}$ centered at any pair of vertices
within $G'$, a contradiction.
\end{proof}

Choose a Gallai partition of $G'$ with the smallest number of parts,
say $q$, and let red and blue be the colors that appear in between
the parts of this partition. Note that from the argument above, red
and blue do not appear on on edges coming from vertices of $T$. Let
$s = |T|$.

If $2\leq q \leq 3$, then by the minimality of $q$, we may assume
that $q=2$. Then if $I_{1}$ and $I_{2}$ are the parts of this
partition, say with $|I_{1}| \geq |I_{2}|$, we have \beqs
|I_{1}|+|I_{2}| & = & |G|-|T|\\
~ & \geq & \frac{5n-6}{2}+(k-3)-s\\
~ &\geq &2n+\frac{n-8}{2}+(k-s-2), \eeqs which means that
$|I_{1}|\geq n+\frac{n-4}{4}\geq n+m$. Then any two vertices of
$I_{2}$ form the centers of a monochromatic copy of $K_{1, m} \cup
K_{1, n}$, a contradiction. Thus, we may assume that $q \geq 4$.

If $s \leq k - 3$, then we may apply the same argument as in the
case $k = 3$ (Cases~\ref{Case:k3t2} and~\ref{Case:k3t4}). Thus,
suppose $s = k - 2$. This means that $|G'| \geq \frac{5n - 8}{2}$.

\begin{claim}\label{Claim:m+1}
Every part has order at least $m + 1$.
\end{claim}

\begin{proof}
First suppose there is a part $H_{0}$ of order $1$. Let $A$ be the
set of vertices with red edges to $H_{0}$ and let $B$ be the set of
vertices with blue edges to $H_{0}$, say with $|A| \geq |B|$. This
means $|A| \geq n + 2m + 1$ so there can be no vertex anywhere other
than the vertex of $H_{0}$ with at least $m$ red edges. This means
that every vertex from $B$ has at least $n + m + 1$ red edges to
$A$, producing a red copy of $K_{1, n} \cup K_{1, m}$.

Thus, suppose there is a part $H$ of order $r$ with $2 \leq r \leq
m$. Then each vertex in $H$ has at least $\frac{|G \setminus H|}{2}
\geq n + m$ edges in a single color to $G \setminus H$ (since $m
\leq \frac{n - 8}{6}$). Choosing two vertices from $H$ produces a
monochromatic copy of $K_{1, n} \cup K_{1, m}$, a contradiction.
\end{proof}

If there is a vertex $v$ with at least $n + 1$ incident red (or
blue) edges to other parts of the Gallai partition, then if we let
$u$ be a vertex with a red edge to $v$, then $v$ is the center of a
red copy of $K_{1, n}$ avoiding $u$ and $u$ is the center of a
disjoint red copy of $K_{1, m}$ since $v$ is in a part of the Gallai
partition with order at least $m + 1$ (by Claim~\ref{Claim:m+1}).
Thus, there is no vertex $v \in G'$ with red or blue degree at least
$n$ to other parts of the Gallai partition.

Applying Lemma~\ref{Lemma:StarStability} with $r = 8$, we see that
the Gallai partition of $G'$ has exactly $5$ parts, each of order at
least $\frac{n - 5}{2}$, and every vertex of $G'$ has at least $n -
5$ incident edges in red and at least $n - 5$ incident edges in blue
to other parts of the Gallai partition.

Let $H_{1}$ be a largest part of this Gallai partition of $G'$.
Since $H_{1}$ contains no rainbow triangle, by
Theorem~\ref{Thm:K3-Star}, there is a monochromatic star on at least
$\frac{2|H_{1}|}{5} \geq m$ edges within $H_{1}$. If this star has
some color $i$ for which $T_{i} \neq \emptyset$, then there is a
monochromatic copy of $K_{1, n} \cup K_{1, m}$ in color $i$ so this
must be either red or blue, say red. This along with the structure
we have already shown implies the existence of a red copy of $K_{1,
n} \cup K_{1, m}$, a contradiction. This completes the proof of
Proposition~\ref{Prop:Small_m_Large_n_Upper}.
\end{proof}

\section{For equal $m$ and $n$} \label{Sec:Equal}

In this section we give a proof for Theorem \ref{Thm:GallaiRamseyLargem}.

For $m=n$, we first give a lower bound for the Gallai Ramsey number of
$K_{1,n}\cup K_{1,n}$.

\begin{lemma}\label{Lemma:EqualLower}
For $k \geq 3$,
$$
{\rm gr}_{k}(K_{3} :K_{1,n}\cup K_{1,n})\geq 3n+k-1.
$$
\end{lemma}

\begin{proof}
We prove this result by inductively constructing a coloring $G_{k}$
of $K_{t}$ where $t=3n+k-2$ which contains no rainbow triangle and
no monochromatic copy of $K_{1,n}\cup K_{1,n}$.

Let $G_3$ be a graph constructed by steps as follows:
\begin{itemize}
\item Let $F_1$ be a complete graph $K_{2n-1}$ edge-decomposed into an $(n-1)$-regular graph of order $2n-1$ with color $1$ and an $(n-1)$-regular graph of order $2n-1$ with color $2$.
\item Let $F_2$ be a complete graph $K_{3n-1}$ obtained from the graph $F_1$ and a complete graph $K_n$ colored with color $1$ by adding all edges from $F_1$ to $K_{n}$ with color $3$.
\item Let $G_3$ be a complete graph $K_{3n+1}$ obtained from $F_2$ by adding two new vertices $v,w$ and adding the edges from $v$ to $F_2$ with color $1$, and the edges from $w$ to $F_2\cup \{v\}$ with color $2$.
\end{itemize}

In order to construct $G_{i + 1}$, we a vertex $v_{i + 1}$ to
$G_{i}$ with all edges from $v_{i + 1}$ to $G_{i}$ having color $i +
1$ for $3 \leq i \leq k-1$. This coloring certainly contains no
rainbow triangle or monochromatic copy of $K_{1,n}\cup K_{1,n}$ and
has order $3n+k-2$, completing the construction.
\end{proof}

\begin{lemma}\label{Lemma:EqualUpper}
For $k \geq 3$,
$$
{\rm gr}_{k}(K_{3} :K_{1,n}\cup K_{1,n})\leq 3n+k-1.
$$
\end{lemma}

\begin{proof}
We will assume $n$ is odd since the even case can be proved
similarly. Suppose $k\geq 3$ and let $G$ be a Gallai coloring of
$K_{3n+k-1}$.

Let $T$ be a maximal set of vertices where each vertex of $T$ has
all edges in a single color to $G \setminus T$ with the additional
assumption that $|G \setminus T| \geq 2n$. Then in order to avoid a
monochromatic copy of $K_{1, n} \cup K_{1, n}$, there is at most one
vertex in $T$ with edges of each color to $G \setminus T$, meaning
that $|T| \leq k$. Let $G' = G \setminus T$ so $|G'| = |G| - |T|
\geq 3n - 1$.

Since $G'$ contains no rainbow triangle, by
Theorem~\ref{Thm:G-Part}, there is a Gallai partition of $G'$, say
using red and blue on edges between the parts of this partition. In
order to avoid a monochromatic copy of $K_{1, n} \cup K_{1, n}$, if
there is a vertex $v_{i} \in T$ with color $i$ on all edges to $G'$,
there is no vertex with at least $n$ incident edges in a single
color $i$ within $G'$. By Lemma~\ref{Lemma:StarStability}, this
means that red and blue cannot appear on edges between $T$ and $G'$
so $|T| \leq k - 2$ and $|G'| \geq 3n + 1$.

Let $H_{1}, H_{2}, \dots, H_{t}$ be the parts of the Gallai
partition of $G'$ chosen so that $t$ is minimized, say with $|H_{i}|
\geq |H_{i + 1}|$ for all $i$. If $2 \leq t \leq 3$, then by
minimality of $t$, we may assume $t = 2$, say with red edges between
$H_{1}$ and $H_{2}$. In this case, if $|H_{2}| = 1$, then the vertex
of $H_{2}$ can be moved to $T$, contradicting the maximality of $T$.
Thus, $|H_{2}| \geq 2$. In order to avoid a red copy of $K_{1, n}
\cup K_{1, n}$, we have $|H_{1}| \leq 2n - 1$, so $|H_{2}| \geq (3n
+ 1) - (2n - 1) = n + 2$. We then find a red copy of $K_{1, n} \cup
K_{1, n}$ centered on one vertex from each of $H_{1}$ and $H_{2}$, a
contradiction. This means that we may assume that $t \geq 4$.

Since all edges between any pair of parts have a single color and a
monochromatic copy of $K_{n + 1, n + 1}$ contains a monochromatic
copy of $K_{1, n} \cup K_{1, n}$, we immediately see that there is
at most one part with at least $n + 1$ vertices.

\setcounter{case}{0}
\begin{case}
There is a part $H_{1}$ with $|H_{1}| \geq n + 1$.
\end{case}

For any other part $H_{i}$ with $2 \leq i \leq t$, let $A_{i}$ be
the set of vertices with red edges to $H_{i}$ and let $B_{i}$ be the
set of vertices with blue edges to $H_{i}$. Certainly $H_{1}$ is in
either $A_{i}$ or $B_{i}$, so the next claim shows the opposite set
is small.

\begin{claim}\label{ClaimnA}
If $H_{1} \subseteq A_{i}$ (or $H_{1} \subseteq B_{i}$), then
$|B_{i}| \leq n - 1$ (resp. $|A_{i}| \leq n - 1$).
\end{claim}

\begin{proof}
Without loss of generality, suppose $H_{1} \subseteq A_{i}$ and for
a contradiction, suppose that $|B_{i}| \geq n$. First we
additionally assume that $|B_{i}| \geq n + 1$. Then by minimality of
$t$, there is a part with blue edges to $H_{1}$, say containing a
vertex $v$. Then choosing any vertex $u \in H_{i}$, we have a blue
copy of $K_{1, n}$ centered at $u$ with edges to $B_{i} \setminus
\{v\}$ (since $v$ might be in $B_{i}$) and a blue copy of $K_{1, n}$
centered at $v$ with edges to $H_{1}$, a contradiction. Thus, we may
assume that $|B_{i}| = n$.

Next suppose that $A_{i} = H_{1}$. If $|H_{i}| = 1$, then $|H_{1}| =
|G'| - 1 - n \geq 2n$. Since $|B_{i}| = n \geq 3$, there are two
vertices $u, v \in B_{i}$ with all one color on their edges to
$H_{1}$. These form the centers of a monochromatic copy of $K_{1, n}
\cup K_{1, n}$, so we may assume that $|H_{i}| \geq 2$. Since
$|H_{i} \cup H_{1}| = |G'| - |B_{i}| \geq 2n + 1$, there can be at
most one vertex in $B_{i}$ with blue edges to $H_{1}$. On the other
hand, by minimality of $t$, there must then be exactly one vertex in
$B_{i}$ with blue edges to $H_{1}$. This means there are exactly $n
- 1$ vertices in $B_{i}$ with all red edges to $H_{1}$. Then if we
choose one vertex $u \in H_{i}$ and one vertex $v \in H_{1}$, then
there is a red copy of $K_{1, n} \cup K_{1, n}$ centered at $u$ with
edges to $H_{1} \setminus \{v\}$ and centered at $v$ with edges to
$B \cup H_{i}$. Thus, we may assume that $A_{i}$ contains at least
one other (smaller) part in addition to $H_{1}$.

Let $H_{j}$ be a part in $A_{i} \setminus H_{1}$. Then in order to
avoid creating a blue copy of $K_{1, n} \cup K_{1, n}$ centered at a
vertex in $H_{i}$ and a vertex in $H_{j}$, all edges from $H_{j}$ to
$H_{1}$ must be red. By minimality of $t$, there is a part $H_{j}
\subseteq A_{i}$ with some red edges to a part $H_{\ell} \subseteq
B_{i}$. Choose a vertex $v \in H_{j}$ and a vertex $u \in H_{i}$.
Choose $n$ red neighbors of $v$ by first selecting all of
$H_{\ell}$, then vertices from $H_{i} \setminus \{u\}$, and finally
some vertices of $H_{1}$ as needed. Let $S$ be the set of vertices
in this red copy of $K_{1, n}$ centered at $v$. Since $|B_{i}| = n$,
we have $|B_{i} \setminus H_{\ell}| \leq n - 1$. Since $|G'| - |S|
\geq 2n$, there are at least $2n - 1 - (n - 1) = n$ remaining red
neighbors of $u$, to form a second disjoint red copy of $K_{1, n}$,
a contradiction.
\end{proof}

If we choose $i = 1$, then in order to avoid a monochromatic copy of
$K_{n + 1, n + 1}$, we immediately see that $|A_{1}|, |B_{1}| \leq
n$. It turns out that we can say a bit more.

\begin{claim}\label{ClaimnB}
The reduced graph of $A_{1}$ is a blue complete graph and the
reduced graph of $B_{1}$ is a red complete graph.
\end{claim}

\begin{proof}
We show that the reduced graph of $A_{1}$ is a blue complete graph
since the other proof is symmetric. Suppose not, so there are red
edges between a pair of parts $H_{i}$ and $H_{j}$ within $A_{1}$.
Choose one vertex from each of these parts, say $v_{i} \in H_{i}$
and $v_{j} \in H_{j}$. Then there is a red copy of $K_{1, n}$
centered at $v_{i}$ with $n$ edges to all of $(H_{j} \setminus
v_{j})$ and part of $H_{1}$. Let $S$ be the vertices of this star.
Since the edges from $H_{1}$ to $H_{j}$ are red, by
Claim~\ref{ClaimnA}, we have $|B_{j}| \leq n - 1$. In $G' \setminus
S$, there are at least $|G'| - |S| - |B_{j}| - 1 \geq (3n + 1) - (n
+ 1) - (n - 1) - 1 = n$ red edges incident to $v_{j}$. This star,
along with $S$, forms a red copy of $K_{1, n} \cup K_{1, n}$, a
contradiction.
\end{proof}

Let $H_{x}$ be a smallest part within $A_{1} \cup B_{1}$, say with
$H_{x} \subseteq A_{1}$. Let $B_{1}'$ be the union of the parts with
red edges to $H_x$ and let $B_{1}''$ be the union of the parts with
blue edges to $H_x$. To avoid a blue $K_{1,n}\cup K_{1,n}$, there
exist a part $H_y$ in $B_{1}''$ and a part $H_z$ in $A_{1}$ such
that the edges from $H_y$ to $H_z$ are red. To avoid a red
$K_{1,n}\cup K_{1,n}$, we have $|B_{1}'|+|H_1|+|H_y|<2n$, and hence
$|B_{1}''|+|A_1|-|H_y|\geq n+1$. Since $H_{x}$ is a smallest part,
it follows that $|B_{1}''|+|A_1|-|H_x|\geq n+1$. Choose $u\in H_x$
and $v\in B_1'$. Then the edges from $u$ to $(B_1''\cup A_1)-H_x$
and the edges $v$ to $H_1$ form a blue $K_{1,n}\cup K_{1,n}$, a
contradiction.


\begin{case}
Every part of the Gallai partition of $G'$ has order at most $n$.
\end{case}

Again for each part $H_{i}$ of the Gallai partition of $G'$, let
$A_{i}$ be the set of vertices with red edges to $H_{i}$ and let
$B_{i}$ be the set of vertices with blue edges to $H_{i}$. For any
pairwise disjoint sets of vertices $V_{1}, V_{2}, \dots, V_{\ell}$
and a disjoint vertex $v$, we say that we \emph{choose-in-order $n$
neighbors of $v$ from $V_{1}, V_{2}, \dots, V_{\ell}$} if we choose
all neighbors of $v$ from $V_{1}$, then from $V_{2}$, and so on
until we have $n$ neighbors of $v$ by using up all of each set
before moving on to the subsequent set. Obviously, this assumes that
$|V_{1} \cup V_{2} \cup \dots \cup V_{\ell}| \geq n$.

\begin{claim}\label{Claimn15}
There exists a part $H_{i}$ such that one of $|A_{i}| \leq n$ or $|B_{i}| \leq n$.
\end{claim}

\begin{proof}
Suppose not, so $|A_{i}| \geq n + 1$ and $|B_{i}| \geq n + 1$ for
all $i$ with $1 \leq i \leq t$. Choose an arbitrary index $i$.

Suppose further that there is a vertex $v \in A_{i}$ with at most $n
- 1$ blue edges to $B_{i}$. Let $H_{j} \subseteq A_{i}$ be the part
containing $v$ and choose $u \in H_{i}$. Since $|A_{j}| \geq n + 1$,
there are at least $n + 1$ red neighbors of $v$. Choose-in-order $n$
neighbors of $v$ using red edges from the sets $H_{i} \setminus
\{u\}, B_{i}, A_{i}$. Let $S$ be the resulting red copy of $K_{1,
n}$. Since we have assumed $|S| = n + 1$, $|H_{i}| \leq n$, and
$|B_{i} \setminus S| \leq n - 1$, we have that $|A' \setminus S|
\geq n$. This means that there remains in $G' \setminus S$ a red
copy of $K_{1, n}$ centered at $u$, a contradiction. Thus, we may
assume that every vertex in $A_{i}$ has at least $n$ blue edges to
$B_{i}$ and similarly every vertex in $B_{i}$ has at least $n$ red
edges to $A_{i}$.

Then the number of red edges plus the number of blue edges between
$A_{i}$ and $B_{i}$ is at least $n\cdot |A_{i}| + n\cdot |B_{i}| =
n(3n + 1 - |H_{i}|)$. On the other hand, the total number of edges
between $A_{i}$ and $B_{i}$ is $|A_{i}||B_{i}| \leq \left( \frac{3n
+ 1 - |H_{i}|}{2} \right)^{2} < n(3n + 1 - |H_{i}|)$, a
contradiction.
\end{proof}

By Claim~\ref{Claimn15}, there is a part $H_{j}$ such that either
$|A_{j}| \leq n$ or $|B_{j}| \leq n$, say $|B_{j}| \leq n$. We now
consider cases based on the value of $|B_{j}|$.

\begin{subcase}
$|B_{j}| \leq n - 1$.
\end{subcase}

First a claim about the red edges incident to a vertex in $A_{j}$.

\begin{claim}\label{Claimn17}
For each vertex $v \in A_{j}$, there are at most $n$ red edges incident to $v$.
\end{claim}

\begin{proof}
Suppose, for a contradiction, that there is a vertex $v \in A_{j}$
with at least $n + 1$ incident red edges and let $u$ be any vertex
in $H_{j}$. Then choose-in-order $n$ vertices with red edges to $v$
from $H_{j} \setminus \{u\}, B_{j}, A_{j}$ and let $S$ be the
resulting red copy of $K_{1, n}$. Then since $|B_{j}| \leq n - 1$,
we have $|A_{j} \setminus S| \geq n$ so there is a second red copy
of $K_{1, n}$ centered at $u$ with edges to $A_{j} \setminus S$, a
contradiction.
\end{proof}

By Claim~\ref{Claimn17}, every vertex in $A_{j}$ has at most $n$
incident red edges and so at least $n + 1$ incident blue edges. Next
a claim about the reduced graph restricted to $A_{j}$.

\begin{claim}\label{Claimn18}
The blue edges in the reduced graph of $G'$ restricted to $A_{j}$ form a union of cliques.
\end{claim}

\begin{proof}
Suppose, for a contradiction, that there are three parts within
$A_{j}$, say $H_{i_{1}}, H_{i_{2}}$, and $H_{i_{3}}$ such that the
edges from $H_{i_{1}}$ to $H_{i_{2}}$ are red while the edges from
$H_{i_{3}}$ to $H_{i_{1}} \cup H_{i_{2}}$ are blue. Note that
$H_{i_{2}} \subseteq A_{i_{1}}$ and $H_{i_{3}} \subseteq B_{i_{1}}$.

If $|A_{i_{1}}| \geq n + 1$ (so each vertex in $H_{i_{1}}$ has at
least $n + 1$ incident red edges to other parts), then let $u \in
H_{j}$ and $v \in H_{i_{1}}$ and choose-in-order $n$ vertices with
red edges to $v$ from $H_{j} \setminus \{u\}, B_{j}, A_{j}$ to form
a red copy of $K_{1, n}$, say $S$. Since $|B_{j} \setminus S| \leq n
- 1$ and $|S| = n + 1$, we have $|A_{j} \setminus S| \geq n$ so
there is a (disjoint) red copy of $K_{1, n}$ centered at $u$, a
contradiction meaning that $|A_{i_{1}}| \leq n$.

Since $|A_{i_{1}}| \leq n$ and $|H_{i_{1}}| \leq n$, we have
$|B_{i_{1}}| \geq n + 1$ (so each vertex in $H_{i_{1}}$ has at least
$n + 1$ blue edges to other parts). Let $u \in H_{i_{1}}$ and $v \in
H_{i_{3}}$ and choose-in-order $n$ vertices with blue edges to $v$
from $H_{i_{1}} \setminus \{u\}, H_{i_{2}}, B_{i_{3}}$ and let $S$
be the resulting blue copy of $K_{1, n}$. Since $H_{i_{2}} \subseteq
A_{i_{1}}$, we have $|A_{i_{1}} \setminus S| \leq n - 1$. This means
that $u$ has at least $|G' \setminus (H_{i_{1}} \cup A_{i_{1}})|
\geq n$ blue edges to $G' \setminus S$. This produces a second
disjoint blue copy of $K_{1, n}$, for a contradiction.
\end{proof}

By Claim~\ref{Claimn18}, the blue edges in the reduced graph of $G'$
restricted to $A_{j}$ form a union of cliques, say $J_{1}, J_{2},
\dots, J_{p}$. If $p \geq 2$, then all edges between pairs of these
cliques must be red. By Claim~\ref{Claimn17} (considering a vertex
in $J_{1}$), $|(A_{j} \setminus J_{1}) \cup H_{j}| \leq n$ and
similarly (by considering a vertex in $J_{2}$) $|J_{1}| < n$ so
$|A_{j}| + |H_{j}| \leq 2n - 1$. Since $|B_{j}| \leq n - 1$, this
means that $|G'| \leq 3n - 2$, a contradiction, meaning that $p = 1$
so we arrive at the following fact.

\begin{fact}\label{Fact:blue}
The reduced graph restricted to $A_{j}$ is a single blue clique.
\end{fact}

\begin{claim}\label{Claimn21}
For any three parts $X_{1}, X_{2} \subseteq A_{j}$ and $Y \subseteq
B_{j}$, if the edges from $Y$ to $X_{1}$ are red, then the edges
from $Y$ to $X_{2}$ are also red. Symmetrically, if the edges from
$Y$ to $X_{1}$ are blue, then the edges from $Y$ to $X_{2}$ are also
blue.
\end{claim}

\begin{proof}
Suppose, for a contradiction, that there exist parts $X_{1}, X_{2}
\subseteq A_{j}$ and $Y \subseteq B_{j}$ with red edges from $Y$ to
$X_{1}$ and blue edges from $Y$ to $X_{2}$. By Fact~\ref{Fact:blue},
the edges from $X_{1}$ to $X_{2}$ are blue. Then the same argument
as in the proof of Claim~\ref{Claimn18} produces a contradiction.
\end{proof}

By minimality of $t$, $A_{j}$ is a single part of the partition, but since $|A_{j}| > n$,
this is a contradiction, completing the proof in this subcase.

\begin{subcase}
$|B_{j}| = n$.
\end{subcase}

The outline of the proof of this subcase is similar to the outline
of the proof of the previous subcase. First an analogue of
Claim~\ref{Claimn18} which follows from exactly the same proof.

\begin{claim}\label{Claimn18b}
The blue edges in the reduced graph of $G'$ restricted to $A_{j}$ form a union of cliques.
\end{claim}

By Claim~\ref{Claimn18b}, the blue edges in the reduced graph of
$G'$ restricted to $A_{j}$ form a union of cliques, say $J_{1},
J_{2}, \dots, J_{p}$. If $p \geq 2$, then all edges between pairs of
these cliques must be red. By the same argument as the proof of
Claim~\ref{Claimn18}, we see that for every part $X \subseteq J_{i}$
and for every part $Y \subseteq B_{j}$, if the edges from $X$ to $Y$
are red, then all the edges from $Y$ to $J_{i} \setminus X$ are red.
By minimality of $t$, this means that each set $J_{i}$ is a single
part of the Gallai partition of $G'$, meaning that the reduced graph
restricted to $A_{j}$ is one red clique.

Note that $|A_{j} \cup H_{j}| = |G'| - n \geq 2n + 1$ so since every
part has order at most $n$, for every part $H_{i} \subseteq A_{j}$
we have $|A_{i}| \geq n$. By minimality of $t$, there is at least
one part $H_{i} \subseteq A_{j}$ with red edges to a part $H_{\ell}
\subseteq B_{j}$. Choose two vertices $u \in H_{i}$ and $v \in
H_{j}$ and choose-in-order $n$ vertices with red edges to $u$ from
$H_{j} \setminus \{v\}, H_{\ell}, A_{i}$ and let $S$ be this red
star. Then since $|H_{\ell}| \geq 1$, $u$ has at least $n$ remaining
incident red edges in $G' \setminus S$, making a disjoint red copy
of $K_{1, n}$, a contradiction completing the proof.
\end{proof}

\section{For general $m$ and $n$} \label{Sec:General}

In this section we give a proof for Theorem \ref{Thm:GallaiRamseyGeneralCase}.

We first give a general lower bound on the Gallai-Ramsey number for
$K_{1,n}\cup K_{1,m}$.

\begin{lemma}\label{Lemma:GeneralLower}
For $k \geq 2$, $m\leq n$,
$$
{\rm gr}_{k}(K_{3} :K_{1,n}\cup K_{1,m})\geq \begin{cases}
2n+m+k-5 & \text{ if $n$ is even,}\\
2n+m+k-4 & \text{ if $n$ is odd.}
\end{cases}
$$
\end{lemma}

\begin{proof}
For odd $n$, we prove this result by inductively constructing a
coloring of $K_{t}$ where $t=2n+m+k-5$ which contains no rainbow
triangle and no monochromatic copy of $K_{1,n}\cup K_{1,m}$. We
construct $G_{3}$ by making four copies of $K_{\frac{n-1}{2}}$ in
color $1$ and one copy of $K_m$ in color $1$, and then inserting
edges of colors $2$ and $3$ between the copies to form a blow-up of
the unique $2$-colored $K_{5}$ which contains no monochromatic
triangle. This coloring clearly contains no rainbow triangle and no
monochromatic copy of $K_{1,n}\cup K_{1,m}$, completing the base
construction.

For even $n$, we prove this result by inductively constructing a
coloring of $K_{t}$ where $t=2n+m+k-6$ which contains no rainbow
triangle and no monochromatic copy of $K_{1,n}\cup K_{1,m}$. We
construct $G_{3}$ by making three copies of $K_{\frac{n-2}{2}}$ in
color $1$ and one copy of $K_m$ in color $1$ and one copy of
$K_{\frac{n}{2}}$ in color $1$, and then inserting edges of colors
$2$ and $3$ between the copies to form a blow-up of the unique
$2$-colored $K_{5}$ which contains no monochromatic triangle. This
coloring clearly contains no rainbow triangle and no monochromatic
copy of $K_{1,n}\cup K_{1,m}$, completing the base construction.

For $i$ with $3 \leq i \leq k - 1$, given $G_{i}$, we construct
$G_{i + 1}$ by adding a single vertex with all edges to $G_{i}$
having color $i + 1$. This coloring certainly contains no rainbow
triangle or monochromatic copy of $K_{1,n}\cup K_{1,m}$ and has the
desired order, completing the construction.
\end{proof}

The lower bounds in Theorem~\ref{Thm:GallaiRamseyGeneralCase} can be
derived from Lemmas \ref{Lemma:Small_m_Large_n_Lower} and
\ref{Lemma:GeneralLower}. To complete the proof of
Theorem~\ref{Thm:GallaiRamseyGeneralCase}, we provide the following
upper bound.

\begin{lemma}\label{Lemma:GeneralUpper}
For $k \geq 2$, $m\leq n$,
$$
{\rm gr}_{k}(K_{3} :K_{1,n}\cup K_{1,m})\leq \begin{cases}
3n+3m+k-3 & \text{ if $n$ is even,}\\
3n+3m+k-2 & \text{ if $n$ is odd.}
\end{cases}
$$
\end{lemma}

\begin{proof}
We only give the proof of the case when $n$ is odd since the even
case can be proved similarly. Suppose $k\geq 3$ and let $G$ be a
Gallai coloring of $K_{N}$ where
$$
N=\begin{cases}
3n+3m+k-3 & \text{ if $n$ is even,}\\
3n+3m+k-2 & \text{ if $n$ is odd.}
\end{cases}
$$

Let $T$ be a maximal set of vertices where each vertex of $T$ has
all edges in a single color to $G \setminus T$ with the additional
assumption that $|G \setminus T| \geq n + m$. Then in order to avoid
a monochromatic copy of $K_{1, n} \cup K_{1, m}$, there is at most
one vertex in $T$ with edges of each color to $G \setminus T$,
meaning that $|T| \leq k$. Let $G' = G \setminus T$ so $|G'| = |G| -
|T| \geq 3n + 3m - 2$.

Since $G'$ contains no rainbow triangle, by
Theorem~\ref{Thm:G-Part}, there is a Gallai partition of $G'$, say
using red and blue on edges between the parts of this partition. In
order to avoid a monochromatic copy of $K_{1, n} \cup K_{1, m}$, if
there is a vertex $v_{i} \in T$ with color $i$ on all edges to $G'$,
there is no vertex with at least $m$ incident edges in a single
color $i$ within $G'$. By Lemma~\ref{Lemma:StarStability}, this
means that red and blue cannot appear on edges between $T$ and $G'$
so $|T| \leq k - 2$ and $|G'| \geq 3n + 3m$.

If $2 \leq t \leq 3$, then by minimality of $t$, we may assume that
$t = 2$. Then one part, say $H_{1}$, has $|H_{1}| \geq
\frac{|G'|}{2} > n + m$. If there are at least $2$ vertices in
$H_{2}$, then there is a monochromatic copy of $K_{1, n} \cup K_{1,
m}$ centered at these two vertices with edges to $H_{1}$ so $|H_{2}|
= 1$. We may then move that vertex in $H_{2}$ to $T$, contradicting
the maximality to $|T|$. We may therefore assume that $t \geq 4$. By
the same argument, we arrive at the following fact.

\begin{fact}\label{Fact:Claim32}
For each $i$ with $1 \leq i \leq t$, we have $|H_{i}| \leq n + m - 1$.
\end{fact}

By Theorem~\ref{Thm:RamseyUnionStars}, we know that $R(K_{1,n}\cup
K_{1,m})=\max\{n+2m, 2n+1,n+m+3\}<3m+3n\leq|G'|$, and hence there
exists a part, say $H_{1}$, with $|H_{1}|\geq 2$. Let $A$ be the set
of vertices with red edges to $H_{1}$ and let $B$ be the set of
vertices with blue edges to $H_{1}$. Since $|G'| \geq 3n + 3m$, one
of $H_{1}, A$, or $B$ must have order at least $\frac{|G'|}{3} \geq
n + m$. By Fact~\ref{Fact:Claim32}, the big set is not $H_{1}$, so
without loss of generality, suppose $|A| \geq n + m$. Then since
$|H_{1}| \geq 2$, these two vertices form the centers of a red copy
of $K_{1,n}\cup K_{1,m}$, a contradiction, completing the proof.
\end{proof}

\end{document}